\documentclass[12pt]{article}
\usepackage{booktabs}
\usepackage{caption}
\usepackage{mathrsfs}
\usepackage{amsmath}
\usepackage{amsfonts,amsthm,amssymb,mathrsfs,bbding}
\usepackage{graphics,multicol}
\usepackage{graphicx}
\usepackage{color}
\usepackage{float}
\usepackage{enumitem}
\usepackage{subfigure}
\usepackage{enumerate}
\usepackage{caption}
\usepackage{rotating}
\usepackage{lscape}
\usepackage{longtable}
\allowdisplaybreaks[4]
\usepackage{tabularx}
\usepackage[colorlinks=true,anchorcolor=blue,filecolor=blue,linkcolor=blue,urlcolor=blue,citecolor=blue]{hyperref}
\usepackage{extarrows}
\usepackage{cite}
\usepackage{tikz}
\usepackage{latexsym,bm}
\usepackage{mathtools}
\evensidemargin=\oddsidemargin\topmargin=-1.5cm

\newtheorem{thm}{Theorem}
\newtheorem{defi}{Definition}

\newtheorem{lem}[thm]{Lemma}
\newtheorem{cor}[thm]{Corollary}

\newtheorem{claim}{Claim}
\theoremstyle{definition}

\addtocounter{section}{0}
\begin{document}

	\title{\bf A lifting theorem for generalized Tur\'an numbers of triangles}

	\author{{Junjie Wang\footnote{\emph{E-mail address:} wangjunjie2827@163.com}, Yaoping Hou\footnote{Corresponding author.}\setcounter{footnote}{-1}\footnote{\emph{E-mail address:} yphou@hunnu.edu.cn}}\\[2mm]
\small College of Mathematics and Statistics, Hunan Normal University,\\
		\small Changsha, Hunan, 410081, China\\[2mm]
		}

	\date{}
	\maketitle
	{\flushleft\large\bf Abstract } For graphs $H$ and $F$, the generalized Tur\'an number $\operatorname{ex}(n,H,F)$ denotes the maximum number of copies of $H$ in an $n$-vertex $F$-free graph. We prove a general lifting principle for the case $H=K_3$ and the forbidden graph is a vertex-disjoint union of several copies of a graph. The key hypothesis is a local neighborhood-forcing condition: there is a graph $R$ with $\operatorname{ex}(n,R)=o(n^2)$ such that $F\subseteq K_1\nabla R$. Under this condition, the corresponding single-forbidden-graph asymptotics, together with a construction attaining the relevant extremal triangle and edge densities simultaneously, lift to an asymptotic value for \(\operatorname{ex}(n,K_3,(s+1)F)\) for every integer \(s \). We also prove an exact version in terms of the maximum value of a weighted expression over all graphs of a given size that avoid the forbidden graph. As applications, we obtain exact or asymptotic results for disjoint unions of suspensions of paths and stars. We also recover known exact results for disjoint odd cycles.
	\begin{flushleft}


	\end{flushleft}

	\section{Introduction}

   A central problem in extremal graph theory asks for the maximum number of edges in an $n$-vertex graph avoiding a forbidden graph. This is the classical Tur\'an number $\operatorname{ex}(n,F)$, where $F$ denotes a forbidden subgraph. One natural extension is to maximize the number of copies of a graph $H$ rather than edges. Given two graphs $H$ and $F$, let $\operatorname{ex}(n,H,F)$ denote the maximum number of copies of $H$ in an $n$-vertex $F$-free graph. When $H=K_2$, this reduces to the ordinary Tur\'an number $\operatorname{ex}(n,F)$. The systematic study of $\operatorname{ex}(n,H,F)$ for general $H$ was initiated by Alon and Shikhelman \cite{AS16}; see also the recent survey of Gerbner and Palmer \cite{GP26}.

We focus on the case $H=K_3$ and on forbidden graphs consisting of several vertex-disjoint copies of a graph. For an integer $r\ge 1$, we write $rF$ for the vertex-disjoint union of $r$ copies of $F$. Thus $\operatorname{ex}(n,K_3,(s+1)F)$ is the maximum number of triangles in an $n$-vertex graph containing no $s+1$ vertex-disjoint copies of $F$. Problems of this type have been studied for several classes of forbidden graphs. Gerbner, Methuku, and Vizer \cite{GMV19} investigated the function $\operatorname{ex}(n,H,kF)$ in the cases where $F$ is a complete graph, a cycle, or a complete bipartite graph. Recently, Hou, Yang, and Zeng \cite{HYZ24} determined the exact value of $\operatorname{ex}(n,K_3,(\ell+1)C_{2k+1})$ for fixed $\ell,k$ and sufficiently large $n$. Their extremal graph is $K_\ell\nabla T_2(n-\ell)$, where $T_2(m)$ denotes the complete balanced bipartite graph on $m$ vertices. Another closely related family of problems concerns suspensions of paths. For a graph $G$, the suspension $\widehat G$ is obtained from $G$ by adding one new vertex adjacent to every vertex of $G$. Mubayi and Mukherjee \cite{MM23} studied triangles in graphs without bipartite suspensions and proved asymptotic results for $\widehat P_k$ when $k=4,5,6$. Mukherjee \cite{M24} later determined the exact value of $\operatorname{ex}(n,K_3,\widehat P_4)$, and Hei, Hou, and Ma \cite{HHM26} determined the exact value of $\operatorname{ex}(n,K_3,\widehat P_5)$ for sufficiently large $n$. In a related direction, suspensions of trees have also been studied from the ordinary Tur\'{a}n perspective; in particular, Zhu, Wang, Zhang, and Zhang \cite{ZWZZ26} obtained edge-extremal results for suspensions of balanced trees.

The proofs of these results share a common local mechanism. If many triangles pass through a vertex $v$, then the induced neighborhood graph $G[N_G(v)]$ has many edges. If \(\operatorname{ex}(n,R)=o(n^2)\), then $G[N_G(v)]$ must contain a copy of $R$; when $F\subseteq K_1\nabla R$, this produces a copy of $F$ using $v$. This phenomenon makes it possible to identify a bounded set of ``core'' vertices which must meet every copy of $F$ in an extremal graph.

Our goal is to formulate this mechanism as a general lifting theorem. We prove both an asymptotic version and an exact version. The asymptotic theorem applies, for instance, to $\widehat P_6$, where the exact single-forbidden-graph problem is not yet fully settled. The exact theorem applies when a suitable weighted single-forbidden-graph problem is known exactly.

\section{Preliminaries and main results}

\subsection{Preliminaries}

We first introduce some notation. The vertex set and edge set of a graph $G$ are denoted by $V(G)$ and $E(G)$, and $e(G)=|E(G)|$. For graphs $H$ and $G$, let $N(H,G)$ denote the number of copies of $H$ in $G$. In particular, $t(G)=N(K_3,G)$ denotes the number of triangles in $G$. For a vertex $v\in V(G)$, let $N_G(v)$ be its neighborhood. The number of triangles containing $v$ is denoted by $t_G(v)$, and satisfies $t_G(v)=e(G[N_G(v)])$. The join of two vertex-disjoint graphs $G$ and $H$ is denoted by $G\nabla H$. The suspension \(\widehat H\) of a graph \(H\) is defined by $\widehat H = K_1\nabla H$.

We first state the local condition used throughout the paper.

\begin{defi}
Let $F$ be a fixed graph. We say that $F$ satisfies the neighborhood-forcing condition if there exists a graph $R$ such that $\operatorname{ex}(n,R)=o(n^2)$ and $F$ is a spanning subgraph of $K_1\nabla R$.

\end{defi}

By the Erd\H{o}s--Stone--Simonovits theorem, every non-bipartite graph \(R\) satisfies \(\operatorname{ex}(n,R)=\Omega(n^2)\) for sufficiently large \(n\). Thus, for sufficiently large \(n\), the condition \(\operatorname{ex}(n,R)=o(n^2)\) implies that \(R\) is bipartite. The examples considered in this paper satisfy this condition. Indeed, $C_{2k+1}\subseteq K_1\nabla P_{2k}$, $\widehat P_r=K_1\nabla P_r$, and $\widehat{K}_{1,r}=K_1\nabla K_{1,r}$.

We now introduce several auxiliary results that will be needed for our main theorems. The first one is due to Hou, Yang, and Zeng \cite{HYZ24}.

\begin{lem}(Hou, Yang, Zeng\cite{HYZ24})\label{HYZ}
    For an integer $k \ge 1$, let $G$ be a $C_{2k+1}$-free graph on $n$ vertices. For sufficiently large $n$, we have
\[
e(G) + t(G) \le \left\lfloor \frac{n^2}{4} \right\rfloor,
\]
and the equality holds if and only if $G = T_2(n)$.
\end{lem}

McLennan \cite{M05} verified the Erd\H{o}s--S\'{o}s conjecture for trees of diameter at most four.

\begin{lem}(McLennan\cite{M05})\label{Mc}
For any tree $T$ with diameter at most four, $\operatorname{ex}(m,T) \le \frac{|T|-2}{2}m.$
\end{lem}

Zhu, Wang, Zhang, and Zhang \cite{ZWZZ26} obtained a sharp bound for \(\operatorname{ex}(n,\widehat T)\) for sufficiently large \(n\), based on the Erd\H{o}s--S\'{o}s conjecture. In particular, their result implies the following result for the suspensions of \(P_4\) and \(P_5\).

\begin{lem}(Zhu, Wang, Zhang, Zhang\cite{ZWZZ26})\label{ZWZZ}
Let $k \in \{ 4,5\}$. We have $\operatorname{ex}(n,\widehat{P}_{k}) = \max\left\{ n_0 n_1 + \left\lfloor \frac{n_0}{2} \right\rfloor : n_0 + n_1 = n \right\}$ for sufficiently large \(n\).
\end{lem}

\begin{lem}[Mukherjee \cite{M24}; Hei, Hou, Ma \cite{HHM26}]\label{MHHM}
Let \(H_n\) be the following \(n\)-vertex graph:
\[
H_n=
\begin{cases}
\text{\(K_{2q,2q}\) with a perfect matching added in one part},
    & n=4q,\\[1mm]
\text{\(K_{2q,2q+1}\) with a perfect matching added in the smaller part},
    & n=4q+1,\\[1mm]
\text{\(K_{2q,2q+2}\) with a perfect matching added in the larger part},
    & n=4q+2,\\[1mm]
\text{\(K_{2q+1,2q+2}\) with a maximum matching added in the larger part},
    & n=4q+3.
\end{cases}
\]
Then, for \(k\in\{4,5\}\) and sufficiently large \(n\), $\operatorname{ex}(n,K_3,\widehat P_k)=\lfloor \frac{n^2}{8}\rfloor$.
Moreover, among all \(n\)-vertex \(\widehat P_k\)-free graphs \(G\) satisfying $t(G)=\operatorname{ex}(n,K_3,\widehat P_k)$, the graph \(H_n\) is the unique one with the maximum number of edges.
\end{lem}

For \(\widehat P_6\), we use the following asymptotic version of Construction $3$ in Hei, Hou and Ma \cite{HHM26}, corresponding to the case \(k=6\). Let \(A\) and \(B\) be two disjoint vertex sets with $|A|=\lceil \frac n2\rceil$, $|B|=\lfloor \frac n2\rfloor$. Start with the complete bipartite graph \(K_{A,B}\). Inside \(A\), add a collection of $\lfloor \frac{|A|}{3}\rfloor$ pairwise vertex-disjoint triangles, leaving at most two vertices of \(A\) uncovered. Denote the resulting graph by \(H'_n\). Then
\[
    e(H'_n)
    =
    |A||B|+3\lfloor \frac{|A|}{3}\rfloor
    =
    \frac{n^2}{4}+O(n),
\]
and
\[
    t(H'_n)
    =
    3\lfloor \frac{|A|}{3}\rfloor |B|
    +
    \lfloor \frac{|A|}{3}\rfloor
    =
    \frac{n^2}{4}+O(n).
\]
It is straightforward to see that \(H'_n\) is \(\widehat P_6\)-free.

\begin{lem}(Mubayi, Mukherjee \cite{MM23})\label{MM}
$\operatorname{ex}(n,K_3,\widehat{P}_6) = \lfloor \frac{n^2}{4} \rfloor + o(n^2)$.
\end{lem}

\begin{lem}(Alon, Shikhelman \cite{AS16})\label{AS}
$\operatorname{ex}(n,K_3,\widehat{K}_{1,r})=o(n^2)$
\end{lem}

\subsection{Main results}

We now state the asymptotic lifting theorem.

\begin{thm}\label{thm::main1}
    Let $F$ be a graph satisfying the neighborhood-forcing condition. Suppose that there exist constants $\alpha\ge 0$, $\beta>0$ such that $\operatorname{ex}(n,K_3,F)=\alpha n^2+o(n^2)$, and $\operatorname{ex}(n,F)=\beta n^2+o(n^2)$. Assume further that, for sufficiently large integer $m$, there exists an $m$-vertex $F$-free graph $Q_m$ satisfying \[ t(Q_m)=\alpha m^2+o(m^2), \qquad e(Q_m)=\beta m^2+o(m^2). \] Then, for every integer $s\ge 0$, \[ \operatorname{ex}\bigl(n,K_3,(s+1)F\bigr) = (\alpha+s\beta)n^2+o(n^2). \]
\end{thm}

For exact results we need a weighted single-forbidden-graph extremal function. For every integer $q\ge 0$, define $\psi_q(m;F) = \max\big\{t(Q) + qe(Q) \,\big| \, |V(Q)| = m,\ Q \text{ is } F\text{-free}\big\}$. 

\begin{thm}\label{thm::main2}
    Let $F$ be a graph satisfying the neighborhood-forcing condition, and let $s$ be an integer with $s \ge 0$. Suppose that, for every $0\le q\le s$, the exact value of $\psi_q(m;F)$ is known and there exist constants $a_0 < a_1 < \cdots < a_s$ such that $\psi_q(m;F) = a_q m^2 + O(m)$ holds for each $q = 0,1,\dots,s$. Then, for sufficiently large $n$,
$$
\operatorname{ex}\big(n,K_3,(s+1)F\big) = \psi_s(n-s;F) + \binom{s}{2}(n-s) + \binom{s}{3}.
$$
Furthermore, if the unique extremal graph for $\psi_s(n-s;F)$ is $Q_{n-s}$ and $Q_{n-s}$ contains no isolated vertices, then the unique extremal graph for $\operatorname{ex}\big(n,K_3,(s+1)F\big)$ is $K_s \nabla Q_{n-s}$.
\end{thm}

    \section{Proof of Theorem \ref{thm::main1}}
\begin{proof}
We first establish the lower bound for $\operatorname{ex}\bigl(n,K_3,(s+1)F\bigr)$. Take an $F$-free graph $Q_{n-s}$ satisfying $t(Q_{n-s}) = \alpha n^2 + o(n^2)$, $e(Q_{n-s}) = \beta n^2 + o(n^2)$. Construct the graph $G_0 = K_s \nabla Q_{n-s}$. Since $Q_{n-s}$ contains no copy of $F$, every copy of $F$ inside $G_0$ must contain at least one vertex from $K_s$. It follows that $G_0$ is $(s+1)F$-free. The number of triangles in $G_0$ satisfies
\[
t(G_0) = t(Q_{n-s}) + s\cdot e(Q_{n-s}) + O_s(n) = (\alpha + s\beta)n^2 + o(n^2). 
\]
Thus, $\operatorname{ex}\bigl(n,K_3,(s+1)F\bigr) \ge (\alpha + s\beta)n^2 + o(n^2)$.

    We proceed by induction on $s$. For the base case $s=0$, the desired identity becomes $\operatorname{ex}(n,K_3,F) = \alpha n^2 + o(n^2)$, and this equality is exactly our standing assumption. Let $s \ge 1$, and suppose the statement holds for $s-1$. Let $G$ be an $n$-vertex $(s+1)F$-free graph with $t(G)=\operatorname{ex}\bigl(n,K_3,(s+1)F\bigr)$. If $G$ is $sF$-free, by the induction hypothesis and $\beta>0$, then 
$$
t(G) \le \big(\alpha + (s-1)\beta\big)n^2 + o(n^2) < (\alpha + s\beta)n^2 + o(n^2), 
$$
which contradicts the lower bound on $\operatorname{ex}\bigl(n,K_3,(s+1)F\bigr)$. We may therefore assume that $G$ contains $s$ vertex-disjoint copies of $F$. Take such a collection of $s$ disjoint copies of $F$, and let $U$ denote the union of their vertex sets. Since $s$ and $F$ are fixed, we have $|U| = O_s(1)$.

Define $W = V(G) \setminus U$. Since $G$ is $(s+1)F$-free, the induced subgraph $G[W]$ is $F$-free. Note that for every vertex $v\in V(G)$, we have $t_G(v) = e\big(G[N(v)]\big)$. Fix a sufficiently small constant $\gamma = \gamma(s,F) > 0$, and define the set of vertices with high triangle-degree by
\[
L = \big\{v \in V(G) : t_G(v) \ge \gamma n^2\big\}. 
\]

\begin{claim}\label{claim::main1}
    $|L| \le s$.
\end{claim}
\begin{proof}
    Suppose for contradiction that $|L| \ge s+1$. Pick distinct vertices $x_1,\cdots,x_{s+1} \in L$. For each index $i$, we have $e\big(G[N(x_i)]\big) = t_G(x_i) \ge \gamma n^2$. Since $F$ satisfies the neighborhood-forcing condition, there exists a graph $R$ such that $F$ is a spanning subgraph of $K_1\nabla R$ and $\operatorname{ex}(n,R) = o(n^2)$.

We inductively select a copy of $R$ inside each $G[N(x_i)]$, with the additional requirement that all chosen copies of $R$ are pairwise vertex-disjoint and avoid all vertices $x_1,\cdots,x_{s+1}$. At each step, the number of vertices we need to exclude is $O_s(1)$, and deleting these vertices removes at most $O_s(n)$ edges from the current neighborhood subgraph. Hence the remaining subgraph still contains at least
\[
\gamma n^2 - O_s(n)
\]
edges. Recalling
\[
\operatorname{ex}(n,R) = o(n^2),
\]
the residual neighborhood subgraph must contain a copy of $R$ whenever $n$ is sufficiently large. Each vertex $x_i$ together with its associated copy of $R$ forms a subgraph isomorphic to $F$. These $s+1$ copies of $F$ are pairwise vertex-disjoint, which yields a contradiction.

We therefore conclude that $|L| \le s$.
\end{proof}

\begin{claim}\label{claim::main2}
    If $|L| \le s-1$, then $\operatorname{ex}\bigl(n,K_3,(s+1)F\bigr) \le (\alpha + s\beta)n^2 + o(n^2)$. 
\end{claim}

\begin{proof}
Let $R_0 = L \cap U$. If $|L| \le s-1$, then $|R_0| \le s-1$. We classify all triangles into three disjoint categories.

\begin{itemize}[labelindent=0pt,leftmargin=*,itemindent=0pt]
\item Triangles whose vertices all lie in $W$. Since the induced subgraph $G[W]$ is $F$-free, we get 
\[
t(G[W]) \le \operatorname{ex}\big(|W|,K_3,F\big) = \alpha n^2 + o(n^2).
\]
\item Triangles that contain at least one vertex $u \in U \setminus L$. The total number of such triangles satisfies
\[
\sum_{u \in U \setminus L} t_G(u) < |U|\gamma n^2 = O_s(\gamma n^2).
\]
\item The remaining triangles intersect $U$, but contain no vertices from $U \setminus L$. Thus any vertex they have inside $U$ must belong to $R_0 = L \cap U$. If such a triangle contains exactly one vertex from $R_0$ and its other two vertices lie in $W$, then the total number of such triangles is at most
\[
|R_0|e(G[W]) \le (s-1)\operatorname{ex}(|W|,F) = (s-1)\beta n^2 + o(n^2).
\]
If such a triangle contains at least two vertices from $R_0$, the total count is only $O_s(n)$.
\end{itemize}

Combining all cases, $t(G) \le \alpha n^2 + (s-1)\beta n^2 + O_s(\gamma n^2) + o(n^2)$. Since $\beta > 0$ and $\gamma > 0$ is taken sufficiently small, we obtain $t(G) \le (\alpha + s\beta)n^2 + o(n^2)$.
\end{proof}

\begin{claim}\label{claim::main3}
    If $|L| = s$, then $G - L$ is $F$-free.
\end{claim}

\begin{proof}
    Let $S = V(G) \setminus L$. Suppose for contradiction that $G[S]$ contains a subgraph $F_0 \cong F$. For every vertex $x \in L$, we have $e\big(G[N(x)]\big) = t_G(x) \ge \gamma n^2$.

By the same argument as in Claim \ref{claim::main1}, we may inductively select a copy of $R$ inside each $G[N(x)]$, with the requirement that each selected copy avoids $F_0$, the set $L$, and all previously chosen copies of $R$. At each step we delete only $O_s(1)$ vertices, so the remaining number of edges satisfies
\[
\gamma n^2 - O_s(n) > \operatorname{ex}(n,R),
\]
which guarantees that such a copy of $R$ can always be found.

Each vertex $x \in L$ together with its corresponding copy of $R$ forms a copy of $F$. These $s$ copies of $F$, together with $F_0$, constitute a copy of $(s+1)F$, yielding a contradiction.

We therefore conclude that $G[S]$ is $F$-free.
\end{proof}

By Claim \ref{claim::main1}, we know that $|L| \le s$. If $|L| \le s-1$, by Claim \ref{claim::main2}, we have $\operatorname{ex}\bigl(n,K_3,(s+1)F\bigr) = (\alpha+s\beta)n^2+o(n^2)$. Assume that $|L|=s$. Combining this with Claim \ref{claim::main3}, we deduce that $G - L$ is $F$-free. 

We partition all triangles in $G$ according to the number of vertices they share with $L$. The number of triangles entirely contained within $S$ is $t(G[S])$. The total number of triangles containing exactly one vertex from $L$ is at most $se(G[S])$. Triangles with exactly two or three vertices in $L$ contribute only $O_s(n)$. Hence
$$
t(G) \le t(G[S]) + s\cdot e(G[S]) + O_s(n).
$$

Since $G[S]$ is $F$-free, $t(G[S]) \le \alpha n^2 + o(n^2)$, $e(G[S]) \le \beta n^2 + o(n^2)$. It follows that $t(G) \le (\alpha + s\beta)n^2 + o(n^2)$. Therefore, $\operatorname{ex}\bigl(n,K_3,(s+1)F\bigr) = (\alpha+s\beta)n^2+o(n^2)$. This completes the proof. 
\end{proof}

\section{Proof of Theorem \ref{thm::main2}} 
\begin{proof}
We first establish the lower bound for $\operatorname{ex}\big(n,K_3,(s+1)F\big)$. Take an $F$-free graph $Q_{n-s}$ attaining the value $\psi_s(n-s;F)$, and construct $G_0 = K_s \nabla Q_{n-s}$. Since $Q_{n-s}$ contains no copy of $F$, every copy of $F$ inside $G_0$ must contain at least one vertex from $K_s$. Consequently, $G_0$ is $(s+1)F$-free. The number of triangles in $G_0$ satisfies
\begin{equation}
    \begin{aligned}
        t(G_0) &= t(Q_{n-s}) + s\cdot e(Q_{n-s}) + \binom{s}{2}(n-s) + \binom{s}{3}\\
        &= \psi_s(n-s;F) + \binom{s}{2}(n-s) + \binom{s}{3}.
        \end{aligned}
    \nonumber
\end{equation}        
This establishes the desired lower bound. 

We prove the upper bound by induction on $s$. For the base case $s=0$, the desired identity becomes $\operatorname{ex}(n,K_3,F) =\psi_0(n;F)=\max\big\{t(Q)\,\big| \, |V(Q)| = n,\ Q \text{ is } F\text{-free}\big\}$, which is immediate. Let $s \ge 1$, and suppose the statement holds for $s-1$. Let $G$ be an $n$-vertex $(s+1)F$-free graph with $t(G)=\operatorname{ex}\bigl(n,K_3,(s+1)F\bigr)$. If $G$ is $sF$-free, then by the induction hypothesis, 
        \begin{equation}
            t(G) \le \psi_{s-1}(n-s+1;F) + \binom{s-1}{2}(n-s+1) + \binom{s-1}{3}. 
            \nonumber
        \end{equation}
Recall that $\psi_s(m;F) = a_s m^2 + O(m)$, $\psi_{s-1}(m;F) = a_{s-1} m^2 + O(m)$, and $a_s > a_{s-1}$. It follows that $ t(G) < \psi_{s}(n-s;F) + \binom{s}{2}(n-s) + \binom{s}{3}$, which contradicts the lower bound on $\operatorname{ex}\bigl(n,K_3,(s+1)F\bigr)$. Therefore, the extremal graph must contain $s$ vertex-disjoint copies of $F$. Let $U$ denote the union of vertices of these copies, and set $W = V(G) \setminus U$. Fix a sufficiently small constant $\gamma = \gamma(s,F) > 0$, and define $L = \{v : t_G(v) \ge \gamma n^2\}$.

By the same argument as in the proof of Theorem \ref{thm::main1}, the neighborhood-forcing condition implies $|L|\le s$. Indeed, if \(L\) contained \(s+1\) vertices, then the induced neighborhoods of these vertices would contain pairwise vertex-disjoint copies of the graph \(R\), yielding \(s+1\) vertex-disjoint copies of \(F\), a contradiction.

We claim that in fact \(|L|=s\). Suppose first that \(|L|\le s-1\). Let \(R_0=L\cap U\). Then \(|R_0|\le s-1\). Repeating the classification of triangles from the proof of Theorem \ref{thm::main1}, the triangles containing at least one vertex of \(U\setminus L\) contribute at most \(O_s(\gamma n^2)\), and all remaining triangles are bounded by
\[
    t(G[W])+|R_0|e(G[W])+O_s(n)
    \le
    \psi_{|R_0|}(|W|;F)+O_s(n).
\]
Since \(\psi_q(m;F)\) is nondecreasing in \(q\) and \(|R_0|\le s-1\), this gives
\begin{equation}
    \begin{aligned}
        t(G)&\le \psi_{s-1}(|W|;F)+O_s(\gamma n^2)+O_s(n)\\
        &= a_{s-1}n^2+O_s(\gamma n^2)+O_s(n).
        \end{aligned}
    \nonumber
\end{equation}   
Since \(\gamma>0\) is chosen sufficiently small and \(a_s>a_{s-1}\), the last inequality contradicts the lower bound for sufficiently large \(n\). Therefore $|L|=s$. 

Let $S = V(G) \setminus L$. If $G[S]$ contains a copy of $F$, we may similarly use each $x\in L$ to find $s$ vertex-disjoint copies of $F$. Together with the copy of $F$ inside $G[S]$, these form a copy of $(s+1)F$, yielding a contradiction. Consequently, $G[S]$ is $F$-free.

We now partition all triangles according to the number of vertices they share with $L$:
\[
t(G) \le t(G[S]) + s\cdot e(G[S]) + \binom{s}{2}(n-s) + \binom{s}{3}.
\]
Since $G[S]$ is $F$-free and $|S|=n-s$, we have 
\[
t(G[S]) + s\cdot e(G[S]) \le \psi_s(n-s;F).
\]
It follows that
\[
t(G) \le \psi_s(n-s;F) + \binom{s}{2}(n-s) + \binom{s}{3}.
\]
This establishes the desired upper bound. Therefore, 
$$t(G) = \psi_s(n-s;F) + \binom{s}{2}(n-s) + \binom{s}{3}.$$ 

If the unique extremal graph for $\psi_s(n-s;F)$ is $Q_{n-s}$, we have $G[S] \cong Q_{n-s}$. Moreover, every term in the triangle count bound must be tight. The number of triangles containing exactly one vertex from $L$ is $s\cdot e(G[S])$. Consequently, for every $x\in L$ and every edge $ab\in E(G[S])$, we must have $xa,xb\in E(G)$.

If $Q_{n-s}$ has no isolated vertices, every vertex in $S$ lies on some edge, so each $x\in L$ must be adjacent to all vertices of $S$. Next consider the number of triangles containing exactly two vertices from $L$. This quantity is bounded above by
\[
\binom{s}{2}(n-s).
\]
Equality forces every pair $x,y\in L$ to be adjacent. Consequently, $G[L] \cong K_s$. We therefore conclude that $G \cong K_s \nabla Q_{n-s}$. This completes the proof.     
\end{proof}

\section{Applications of Theorems \ref{thm::main1} and \ref{thm::main2}}

In this section, we present several applications of Theorems \ref{thm::main1} and \ref{thm::main2}. We first consider disjoint unions of suspensions of paths and stars.

\begin{cor}\label{Application2}
   Let \(k\in\{4,5\}\) and let \(s\ge 1\) be an integer. For sufficiently large $n$, we have
$$
\operatorname{ex}\big(n,K_3,(s+1)\widehat{P}_k\big) = t(K_s \nabla H_{n-s}),
$$
and $K_s \nabla H_{n-s}$ is the unique extremal graph. 
\end{cor}
\begin{proof}
    Recall that $k \in \{4,5\}$. Let $R = P_k$. Since $\widehat{P}_k = K_1 \nabla P_k$, and $\operatorname{ex}(n,P_k) = O(n)$, the neighborhood-forcing condition holds. Let $Q_{n-s}$ be an $(n-s)$-vertex $\widehat{P}_k$-free graph. By Lemma \ref{MHHM}, for any integer $q \ge 1$, we have
\begin{equation}
t(Q_{n-s})+qe(Q_{n-s})\le \operatorname{ex}(n-s,K_3,\widehat{P}_k)+q\cdot \operatorname{ex}(n-s,\widehat{P}_k)= t(H_{n-s})+q\cdot e(H_{n-s}), 
\nonumber
\end{equation}
and the above equality holds if and only if $Q_{n-s} = H_{n-s}$. Thus, $\psi_s(n-s;\widehat{P}_k)=t(H_{n-s})+s\cdot e(H_{n-s})$, and the unique extremal graph for $\psi_s(n-s;\widehat{P}_k)$ is $H_{n-s}$. By Theorem \ref{thm::main2}, 
$$
\operatorname{ex}\big(n,K_3,(s+1)\widehat{P}_k\big) = \psi_s(n-s;\widehat{P}_k) + \binom{s}{2}(n-s) + \binom{s}{3},
$$
and the unique extremal graph is $K_s \nabla H_{n-s}$.
\end{proof}

\begin{cor}\label{Application3}
For integers $s \ge 0$, we have
$$
\operatorname{ex}\big(n,K_3,(s+1)\widehat{P}_6\big) = \frac{s+1}{4}n^2 + o(n^2)
$$
for sufficiently large $n$. 
\end{cor}
\begin{proof}
    Let $R = P_{6}$. Since $\widehat{P}_6 = K_1 \nabla P_6$, and $\operatorname{ex}(n,P_{6}) = O(n)$, the neighborhood-forcing condition holds. 

By the classical Erd\H{o}s–Stone–Simonovits theorem, we have $\operatorname{ex}(n,\widehat{P}_6) = \frac{n^2}{4} + o\left(n^2\right)$. Note that $t(H_{n}')=\frac{n^2}{4}+O(n)$, and $e(H_{n}')=\frac{n^2}{4}+O(n)$. We therefore take $\alpha = \frac{1}{4}$, $\beta = \frac{1}{4}$ in Theorem \ref{thm::main1}, which yields $\operatorname{ex}\big(n,K_3,(s+1)\widehat{P}_6\big) = \frac{s+1}{4}n^2 + o(n^2)$.
\end{proof}

\begin{cor}\label{Application4}
For integers \(r\ge 1\) and \(s\ge 0\), we have
$$
\operatorname{ex}\big(n,K_3,(s+1)\widehat{K}_{1,r}\big)=\frac{s}{4}n^2+o(n^2).
$$
\end{cor}
\begin{proof}
    Let \(R=K_{1,r}\). Since $\widehat{K}_{1,r}=K_1\nabla K_{1,r}$, and \(\operatorname{ex}(n,K_{1,r})=o(n^2)\), the neighborhood-forcing condition holds.

We now verify the assumptions of Theorem \ref{thm::main1}. Since \(\chi(\widehat{K}_{1,r})=3\), the Erd\H{o}s–Stone –Simonovits theorem gives $\operatorname{ex}(n,\widehat{K}_{1,r})=\frac{n^2}{4}+o(n^2)$. By Lemma \ref{AS}, we have $\operatorname{ex}(n,K_3,\widehat{K}_{1,r})=o(n^2)$. Thus we may take  \(\alpha=0\) and \(\beta=\frac14\).

Finally, let \(Q_m=T_2(m)\). Since \(T_2(m)\) is bipartite, it is \(\widehat{K}_{1,r}\)-free.
Moreover,
$$
t(Q_m)=0=o(m^2),~e(Q_m)=\left\lfloor \frac{m^2}{4}\right\rfloor=\frac{m^2}{4}+O(1).
$$
Hence \(Q_m\) simultaneously attains the required triangle and edge densities
with
$$
\alpha=0,~\beta=\frac14.
$$
By Theorem \ref{thm::main1}, we obtain
$$
\operatorname{ex}\big(n,K_3,(s+1)\widehat{K}_{1,r}\big)=\left(0+s\cdot\frac14\right)n^2+o(n^2)=\frac{s}{4}n^2+o(n^2).
$$
This completes the proof.
\end{proof}

The following result was proved by Hou, Yang, and Zeng \cite{HYZ24}. We show that it also follows from Theorem \ref{thm::main2} and Lemma \ref{HYZ}.

\begin{cor}\label{Application1}
    For integers $s \ge 1$ and $k \ge 1$, we have
$$
\operatorname{ex}\bigl(n, K_3, (s + 1) C_{2k+1}\bigr) = s\left\lfloor \frac{(n - s)^2}{4} \right\rfloor + (n - s)\binom{s}{2} + \binom{s}{3}
$$
for sufficiently large $n$, and $K_s \nabla T_2(n - s)$ is the unique extremal graph.
\end{cor}
\begin{proof}
    Let $R = P_{2k}$. Since $C_{2k+1} \subseteq K_1 \nabla P_{2k}$, and $\operatorname{ex}(n,P_{2k}) = O(n)$, the neighborhood-forcing condition holds.

Let $Q_{n-s}$ be an $(n-s)$-vertex $C_{2k+1}$-free graph. By Lemma \ref{HYZ}, for any integer $q \ge 1$,
\[
t(Q_{n-s}) + q\cdot e(Q_{n-s}) = \big(t(Q_{n-s})+e(Q_{n-s})\big)+(q-1)e(Q_{n-s}) \le q\left\lfloor \frac{m^2}{4} \right\rfloor.
\]
The equality holds if and only if $Q_{n-s} = T_2(n-s)$. Hence $\psi_q(n-s;C_{2k+1}) = q\lfloor \frac{(n-s)^2}{4} \rfloor$. Substitute this into Theorem \ref{thm::main2}. For $s \ge 1$, we obtain
\[
\operatorname{ex}\big(n,K_3,(s+1)C_{2k+1}\big) = s\left\lfloor \frac{(n-s)^2}{4} \right\rfloor + \binom{s}{2}(n-s) + \binom{s}{3}.
\]
The unique extremal graph is $K_s \nabla T_2(n-s)$.
\end{proof}

\section*{Acknowledgements}

This work was supported by the National Natural Science Foundation of China (No. 12331012) and the Hunan Provincial Innovation Foundation for Postgraduate (No. CX20250747).

\end{document}